\documentclass[13pt,a4paper, oneside]{amsart}
\usepackage{fullpage}
\usepackage{mathrsfs}
\usepackage{CJK}
\usepackage{amsfonts}
\usepackage{xy}
\usepackage{cite}
\usepackage{amssymb}
\usepackage{color, latexsym}
\usepackage{graphicx}
\usepackage{amsbsy,amscd}
\usepackage{fancyhdr}
\usepackage[colorlinks=true,
           citecolor=blue,
           linkcolor =red]{hyperref}
\xyoption{all}
\allowdisplaybreaks

\fontsize{9pt}{9pt}\selectfont
\def\){\big)}
\def\({\big(}
\def\bmu{\boldsymbol{\mu}}
\def\bgl{\boldsymbol{\lambda}}
\def\bnu{\boldsymbol{\nu}}

\def\ga{\alpha}
\def\gb{\beta}

\def\gma{\gamma}
\def\gl{\lambda}

\def\bz{\mathbb{Z}}

\def\fs{\mathfrak{s}}
\def\ft{\mathfrak{t}}
\def\fu{\mathfrak{u}}
\def\fv{\mathfrak{v}}
\def\sa{\mathscr{A}}
\def\h{\mathbf{H}}
\def\sr{\mathscr{R}}
\def\shape{\mathrm{Shape}}
\def\res{\mathrm{res}}
\def\std{\mathrm{Std}}

\def\mpn{\mathscr{P}_{m,n}}

\newdimen\hoogte    \hoogte=14pt    
\newdimen\breedte   \breedte=14pt   
\newdimen\dikte     \dikte=0.5pt    
\newenvironment{point}[2]%
  {\vspace{1\jot}\ifx*#2\let\pointlabel\relax\else\def\pointlabel{#2}\fi
   \refstepcounter{equation}\trivlist
   \item[\hskip\labelsep\theequation.
         \ifx\pointlabel\relax\else\space\pointlabel\space\fi]
   \ignorespaces #1
  \vspace{1\jot}}{\relax}

\numberwithin{equation}{section}

\swapnumbers
\newtheorem{theorem}[equation]{Theorem}
\newtheorem{lemma}[equation]{Lemma}
\newtheorem{proposition}[equation]{Proposition}

\theoremstyle{definition}
\newtheorem{definition}[equation]{Definition}

\newtheorem{notations}[equation]{Notations}

\theoremstyle{remark}
\newtheorem{remark}[equation]{Remark}

\begin{document}
\begin{CJK*}{GBK}{song}
\setlength{\itemsep}{-0.25cm}
\fontsize{13}{\baselineskip}\selectfont
\setlength{\parskip}{0.4\baselineskip}
\vspace*{0mm}
\title[Fusion procedure]{\fontsize{9}{\baselineskip}\selectfont Fusion procedure for Degenerate cyclotomic Hecke algebras}
\author{Deke Zhao}
\address{School of Applied Mathematics, Beijing Normal University at Zhuhai, Zhuhai, 519087, China}

\email{ deke@amss.ac.cn}
\thanks{Zhao is  supported by the National Natural Science Foundation of China (Grant No. 11101037).}
\author{Yanbo Li}
\address{School of Mathematics and Statistics, Northeastern
University at Qinhuangdao, Qinhuangdao, 066004, China}
\email{ liyanbo707@163.com}
\thanks{Li is supported by the Natural Science Foundation of Hebei Province, China (A2013501055).}
\vspace*{-3mm}
\begin{abstract}The primitive idempotents of the generic degenerate cycloctomic Hecke algebras are derived by consecutive evaluations of a certain rational function. This rational function depends only on the Specht modules and the normalization factors are the weights of the Brundan-Kleshchev trace.
\end{abstract}
\maketitle
\vspace*{-5mm}
\section{Introduction}
In \cite{Jucys} Jucys observed  that the primitive idempotents of the symmetric groups can be obtained by taking certain limit values of a rational function in several variables. A similar construction, now commonly referred to as the \textit{fusion procedure}, was developed by Cherednik \cite{Cherednik86}, while complete proofs relying on $q$-version of the Young symmetrizers were given by Nazarov  \cite{Nazarov-98}. This method has already been used by Nazarov (and Tarasov) \cite{Nazarov-2003,Nazarov-2004,Nazarov-Tarasov} and Grime \cite{Grime}. A simple version of the fusion procedure for the symmetric group was given by Molev in \cite{Molev}. Here the idempotents are obtained by consecutive evaluations of a certain rational function. This version of the fusion procedure relies on the existence of a maximal commutative subalgebra generated by the Jucys-Murphy elements and was developed for various algebras and groups (see e.g. \cite{Isaev-Molev-Oskin-2008,Isaev-Molev,Isaev-Molev-O,
Isaev-Molev-O-MBW,Ogievetsky-Andecy2011,Ogievetsky-Andecy2013}).

Let $m,n$ be positive integers and let $W_{m,n}$ be the complex reflection group of type $G(m,1,n)$. By \cite{shephard-toda}, $W_{m,n}$ has a presentation with generators
  $\tau, s_1, \dots, s_{n-1}$ where the defining relations are $\tau^m=1, s_1^2=\cdots=s_{n-1}^2=1$ and the homogeneous relations
  \begin{align*}&\tau s_1\tau s_1=s_1ts_1\tau,&&\tau s_i=s_it\text{ for }i\geq 2,\\
 &s_is_{i+1}s_i=s_{i+1}s_{i}s_{i+1} \text{ for }i\geq1&& s_is_j=s_js_i, \text{ if } |i-j|>1.\end{align*}
 It is well-known that $W_{m,n}\cong(\mathbb{Z}/m\bz)^{n}\rtimes \mathfrak{S}_{n}$, where  $\mathfrak{S}_{n}$ is the symmetric group of degree $n$ (generated by  $s_1, \dots, s_{n-1}$).
 The degenerate cyclotomic Hecke algebra $\h_{m,n}(Q)$ with parameters $Q$ is a natural deformation of the group algebra $KW_{m,n}$ over a filed $K$ (see Definition~\ref{Def:DCHA} below).

 Let  $\mathscr{P}_{m,n}$ be  the set of all $m$-multipartition of $n$. For $\bgl\in\mpn$,  let $S^{\bgl}$ be the Specht module $\h_{m,n}(Q)$  of type $\bgl$. Then $\{S^{\bgl}|\bgl\in\mpn\}$ is the set of irreducible representations of the generic cyclotomic Hecke algebras $\h_{m,n}(Q)$.  Let $E_{\ft}$ be the primitive idempotent of $\h_{m,n}(Q)$ corresponding to the standard $\bgl$-tableau $\ft$. A complete system of pairwise orthogonal primitive idempotents of $\h_{m,n}(Q)$ is parameterized by the set of standard tableaux of the $m$-multipartitions of $n$.

 This paper is concerned with the fusion procedure for $H_{m,n}(Q)$.
 As in \cite{Molev,Ogievetsky-Andecy2013}, we use the Jucys-Murphy elements of $\h_{m,n}(Q)$. The main result of the paper is the following:

\noindent\textbf{Theorem.} \textit{Let $\bgl$ be an $m$-multipartition of $n$ and $\ft$ a standard $\bgl$-tableau. Then the primitive idempotent $E_{\ft}$ of $\h_{m,n}(Q)$ corresponding to $\ft$ can be obtained by the following consecutive evaluations}
\begin{align*}
 E_{\ft}= \Theta_{\bgl}(Q)\Phi(z_1, \cdots, z_n)\biggl.\biggr|_{z=\res_{\ft}(1)}\cdots\, \biggl. \biggr|_{z=\res_{\ft}(n\!-\!1)}\biggl.\biggr|_{z=\res_{\ft}(n)},
\end{align*}
 \textit{where $\Phi(z_1, \cdots, z_n)$ is a rational function in several variables with values in $\h_{m,n}(Q)$ and $\Theta_{\bgl}(Q)$ is a rational functions in variables $Q$.}

 \vspace{2\jot}

 Remarkably, the coefficient $\Theta_{\bgl}(Q)$ appearing in Theorem depends only on the $m$-multipartition $\bgl$. In fact, the coefficient  $\Theta_{\bgl}(Q)$ is the weight of the Brundan-Kleshchev trace on $\h_{m,n}(Q)$ corresponding to the Specht modules $S^{\bgl}$, that is, $\Theta_{\bgl}(Q)$ is the inverse of the Schur element $s_{\bgl}(Q)$ of $\h_{m,n}(Q)$  (see \cite[Theorem~4.2]{Zhao-symm}).

In additional, the degenerate cyclotomic Hecke algebra is a cellular algebra with Jucys-Murphy  element (see \cite{AMR, Mathas-J-Reine}). It may be surprising that the cellularity is not used in the construction of the rational function $\Phi(z_1, \cdots, z_n)$ appearing in Theorem. It seems reasonable that we may develop an abstract framework for the fusion procedure for those algebras equipped with a family of inductively defined Jucys-Murphy elements satisfying some certain conditions (cf. \cite{Mathas-J-Reine}, \cite{Chlo-d'Andecy}).
 We hope to return this issue in future work.

This paper is organized as follows. Section~\ref{Sec:Preliminaries} contains definitions and notations about the multipartitions, the degenerate cyclotomic Hecke algebra, the Jucys-Murphy elements  and the Baxterized elements and gives facts. The rational function $\Theta_{\bgl}(Q)$ is    introduced and investigated in Section~\ref{Sec:Combinatorial-lemmas}, in particular, a combinatorial formulation of $\Theta_{\bgl}(Q)$ is presented. Finally,  we define the rational function $\Phi(z_1, \cdots, z_n)$ and prove the main theorem in the last section.
\subsection*{Acknowledgements} The authors are grateful to Professor Chengming Bai for his hospitality during their visits to the Chern Institute of Mathematics (CIM) in Nankai University. Part of this work was carried out while the first author was visiting CIM and the Kavli Institute for Theoretical Physics China (KITPC) at the Chinese Academy of Sciences in Beijing.

\section{Preliminaries}\label{Sec:Preliminaries}

\begin{point}{}* A {\it partition} $\lambda=(\lambda_1,\lambda_2,\cdots)$ is a decreasing sequence of non-negative integers containing only finitely many non-zero terms.
We define the {\it length} of $\lambda$ to be the smallest integer $\ell(\lambda)$
such that $\lambda_i=0$ for all $i>\ell(\lambda)$ and set $|\lambda|:=\sum_{i \geq 1}\lambda_i$.
If $|\lambda|=n$ we say that $\lambda$ is a \textit{partition} of $n$.

Let $m, n$ be positive integers. An {\it $m$-multipartition} of $n$ is an ordered $m$-tuple
$\bgl=(\gl^1; \dots; \gl^m)$ of partitions $\lambda^{i}$ such that
$n=\sum_{i=1}^m|\lambda^{i}|$. We denote by $\mpn$ the set of all $m$-multipartitions of $n$.

The {\it diagram} of an $m$-multipartition $\bgl$ is the set
$$ [\bgl]:=\{(i,j,c)\in\bz_{>0}\times\bz_{>0}\times \mathbf{m}|1\le j\le\lambda^c_i\}, \quad\text{ where }\mathbf{m}=\{1, \dots, m\}.$$
The elements of $[\bgl]$ are the \textit{nodes} of $\bgl$; more generally, a node is any element of $\bz_{>0}\times \bz_{>0}\times \mathbf{m}$. We may and will identity an $m$-multipartition with its diagram.   A node
$\ga\notin\bgl$ is {\it addable} for $\bgl$ if
$\bgl\cup\{\ga\}$ is the diagram of an $m$-multipartition and a node
$\gb$ of $\bgl$ is {\it removable} for $\bgl$ if
$\bgl\setminus\{\gb\}$ is the diagram of an $m$-multipartition. We denote by $\sa(\bgl)$ (resp. $\sr(\bgl)$) the set of addable (resp. removable) nodes of $\bgl$.

 A {\it $\bgl$-tableau} is a bijection $\ft: [\bgl]\rightarrow\{1,2,\dots,n\}$ and  write $\text{Shape}(\ft)=\bgl$ if $\ft$ is a $\bgl$-tableau.
We may and will identify a $\bgl$-tableau $\ft$ with an $m$-tuple of tableaux
$\ft=(\ft^1; \dots; \ft^m)$, where the {\it$c$-component} $\ft^{c}$ is a $\lambda^{c}$-tableau for $1\le c\le m$.
A $\bgl$-tableau is {\it standard} if in each component the entries increase along the rows and
down the columns. We denote by $\std(\bgl)$ the set of all standard
$\bgl$-tableaux.\end{point}

\begin{definition}\label{Def:DCHA}Let $K$ be a field and $Q=\{q_1, \dots, q_m\}\subset K$.
The {\it degenerate cyclotomic Hecke algebra} is the unital
associative $K$-algebra  $\h\!:=\h_{m,n}(Q)$ generated by
$t,t_1,\dots,t_{n-1}$ and subjected to relations

\begin{enumerate}\item
$(t-q_1)\dots(t-q_m)=0$,
\item $t(t_1tt_1+t_1)=(t_1tt_1+t_1)t$ and $tt_i=t_it$ for $i\geq 2$,
\item $t_i^2=1$ \text{ for } $1\le i< n$,
\item $t_it_{i+1}t_i=t_{i+1}t_it_{i+1}$ for  $1\le i<n-1$,
\item $t_it_j=t_jt_i$ for $|i-j|>1$.
\end{enumerate}
The \textit{Jucys-Murphy elements} of the algebra $\h$ are define inductively as following:
\begin{equation}\label{Equ:JM-def}
J_1\!:=t \quad \text{ and }\quad J_{i+1}\!:=t_iJ_it_i+t_i, \quad i=1, \cdots, n-1.
  \end{equation}\end{definition}
  Then $t_iJ_j=J_jt_i$ if $i\neq j-1, j$ and $J_jJ_k=J_kJ_j$ if $1\le j, k\le n$.
Furthermore, the Jucys-Murphy elements $J_i$  $(i=1, \cdots, n)$ generate a maximal commutative subalgebra  of $\h$, furthermore, the center $\mathcal{Z}(\h)$ of $\h$ is the algebras generated by the symmetric polynomials of the Jucys-Murphy elements $J_1, \cdots, J_n$ (ref. \cite{Brundan}).

For any distinct $i,j=1, \cdots, n-1$, we define the \textit{Baxterized elements} with spectral  parameters $x,y$:
\begin{equation}\label{Equ:t_{ij}(x,y)}
  t_{ij}(x,y):=\frac{1}{x-y}+t_{ij},
\end{equation}
Here and in what follows $t_{ij}$ denotes the element $t_{ij}:=t_it_{i+1}\cdots t_j$ for all $|i-j|>1$.

These Baxterized elements take values in the algebra $K\mathfrak{S}_{n}$ and satisfy the following relations:
\begin{align*}
 & t_i(x,y)t_{i+1}(x,z)t_i(y,z)=t_{i+1}(y,z)t_i(x,z)t_{i+1}(x,y),\\ &t_{ij}(x,y)t_{ji}(y,x)=1-(x-y)^{-2}.
\end{align*}

From now on we let $f(z)=(z-q_1)(z-q_2)\cdots(z-q_m)$ and define the following rational function with values in $\h$:
\begin{equation}\label{Equ:def-t(z)}
  t(z):=\frac{f(z)}{z-t}.
\end{equation}
Then $t(z)$ is a polynomial function in $z$. Moreover the elements $t(z)$ and $t_1(x,y)$ satisfy the following reflection equation with parameters $x, y$:
\begin{equation}\label{Equ:Reflection-equation}
 t(x)t_1(x,y)t(y)t_1+t(x)t_1(x,y)=t_1(x,y)t(x)+t_1t(y)t_1(x,y)t(x).
\end{equation}
\begin{remark}
  The facts that Baxterized elements satisfy the reflection equation~(\ref{Equ:Reflection-equation}) will not be used in this paper. It seems likely that we may use this fact to construct the Bethe subalgebra of degenerate cyclotomic Hecke algebras  (cf. \cite{Isaev-Kirillov}). We hope to return this issue in the future.
\end{remark}

We shall work with a generic degenerated cyclotomic  Hecke algebra, that is, $q_1, \cdots, q_m$ are indeterminates and the algebra $\h$ over a certain localization of the ring $K[q_1, \cdots, q_m]$, or in a specialization such that the following separation condition is satisfied:

\begin{equation*}\label{Equ:Separation}P_\h(Q)=n!
  \prod_{1\le i<j\le m}\prod_{|d|\leq n}(d+q_i-q_j)\neq 0.\end{equation*}
\section{Combinatorial formulae}\label{Sec:Combinatorial-lemmas}
\begin{point}{}* Let $\bgl$ be an $m$-multipartition  of $n$ and $\ft$ a standard $\bgl$-tableau. For $i=1, \cdots, n$, we define the {\it residue} of $i$ in $\ft$ to be $\res_\ft(i)=b-a+q_c$ if the number $i$ appears in the node $(a, b,c)$. More generally, if $x=(a,b,c)$ is in $\bgl\cup \sa(\bgl)$ then we put $\res(x)=b-a+q_c$.\end{point}

The following Lemma is well-known (cf. \cite[Lemma 3.34]{Mathas-book}, \cite[Lemma~3.12]{James-Mathas}), and easy verified by induction on $n$.

\begin{lemma}\label{Lemm:Separation}
Assume that $\h$ is generic.  Let $\bgl$ and $\bmu$ be $m$-multipartitions of $n$ and
  suppose that $\fs\in\std(\bgl)$ and $\ft\in\std(\bmu)$.
\begin{enumerate}
\setlength\itemsep{1pt}
\item $\fs=\ft$ (and $\bgl=\bmu$) if and only if $\res_\fs(k)=\res_\ft(k)$ for
$k=1,\dots,n$.
\item Suppose that $\bgl=\bmu$ and there exists an $i$ such that $\res_{\fs}(k)=
\res_{\ft}(k)$ for all $k\neq i, i+1$. Then either $\fs=\ft$ or  $\fs=\ft(i,i+1)$.
\end{enumerate}
\end{lemma}

\begin{point}{}*
The \textit{conjugate} of  a partition $\gl$ is the partition $\hat{\gl}=(\hat{\gl}_1\geq\hat{\gl}_2\geq\dots)$
 whose  diagram is the transpose of the diagram of $\gl$, i.e. $\hat{\gl}_i$ is the number of nodes in the $i$-th column of the diagram of $\gl$. Hence $\hat{\gl}_1=\ell(\gl)$ and a node
$(i,j)$ of $\gl$ is removable if and only if $j=\gl_i$ and $i=\hat{\gl}_{j}$.

Recall that the $(i,j)$-th {\it hook} in $\lambda$ is the
collection of nodes to the right of and below the node $(i,j)$,
including the node $(i,j)$ itself, and that the $(i,j)$-th {\it hook length} of $\gl$ is
$$h^{\lambda}_{i,j}=\lambda_i-i+\hat{\lambda}_j-j+1.$$

Now let $\gl$ and $\mu$ be partitions.
If $(i, j)$ is a node of $\gl$ then
 \textit{the generalized hook length} of the node $(i, j)$ with respect to $(\gl, \mu)$ is $$h_{i,j}^{\gl,\mu}=\gl_i-i+\hat{\mu}_j-j+1.$$
  Observe that if $\gl=\mu$ then $h_{i,j}^{\gl,\mu}=h_{i,j}^{\gl}$.
\end{point}

\vspace{1\jot}
Let $\bgl=(\gl^1;\cdots; \gl^m)$ be an $m$-multipartition of $n$.  We introduce the following function in variables $Q$:
\begin{equation}\label{Equ:theta-bgl}
 \Theta_{\bgl}(Q)=\prod_{1\leq s\leq m}\prod_{(i,j)\in\gl^s}\prod_{1\leq t\leq m}\frac{1}{h^{\gl^s,\gl^t}_{i,j}+q_s-q_t}.
\end{equation}
Thanks to Lemma~\ref{Lemm:Separation}, $\Theta_{\bgl}(Q)$ is a rational function and  can be reformulated as following:
\begin{equation*}
 \Theta_{\bgl}(Q)=\prod_{1\leq s\leq m}\prod_{(i,j)\in\gl^s}\(\frac{1}{h^{\gl^s}_{i,j}}\prod_{1\leq t\leq m\,\&\,t\neq s}\frac{1}{h^{\gl^s,\gl^t}_{i,j}+q_s-q_t}\).
\end{equation*}
\begin{remark}
  The rational function $\Theta_{\bgl}(Q)$ is the weight of the Brundan-Kleshchev trace on $\h_{m,n}(Q)$ corresponding to the Specht module $S^{\bgl}$, that is, it is the inverse of the Schur element $s_{\bgl}(Q)$ of $\h_{m,n}(Q)$ (ref. \cite[Theorem~5.5]{Zhao} and \cite[Theorem~4.2]{Zhao-symm}).
\end{remark}

\begin{lemma}\label{Lemm:Theta-gl-mu=nu}Let $\gl$ and $\mu$ be partitions. Assume that $\ga=(\imath,\jmath)$ is a removable node of $\gl$ and let $\nu=\gl\!-\!\{\ga\}$. We define the  following rational function in variables $x, y$:
\begin{equation*}
  \Theta_{\gl,\mu}(x, y):=\prod_{(i,j)\in\mu}\frac{1}{h^{\gl,\mu}_{i,j}+x-y}
  \prod_{(i,j)\in\gl}\frac{1}{h^{\mu,\gl}_{i,j}+y-x}.
\end{equation*}
Then \begin{equation}\label{Equ:theta-lambda-mu-nu-z}
  \frac{\Theta_{\gl,\mu}(x,y)}{\Theta_{\nu,\mu}(x,y)}=\prod_{\gb\in\sr(\mu)}\(\res(\gb)-\res(\ga)+y-x\)
  \prod_{\gma\in\sa(\mu)}\(\res(\gma)-\res(\ga)+y-x\)^{-1}.
\end{equation}
\end{lemma}
\begin{proof} Let $z:=x-y$ and suppose that \begin{align*}
    \mu=\(\mu_1=\cdots=\mu_{i_1}>\mu_{i_1+1}=\cdots=\mu_{i_2}>\cdots\cdots>\mu_{i_{p-1}+1}=\cdots=\mu_{i_p}>0\).
  \end{align*}
  We set $i_0=0$, $\mu_{i_{p+1}}=0$ and let $z=q_c-q_t$. Then
  \begin{align*}
   \sr(\mu)&=\left\{(i_k,\mu_{i_k})\mid k=1,\cdots, p \right\}; \\ \sa(\mu)&=\left\{(i_k+1, \mu_{i_{k+1}}+1)\mid k=0, \cdots, p\right\}.\end{align*}
   Therefore, the right hand side of (\ref{Equ:theta-lambda-mu-nu-z}) is
   \begin{align*}
     \mathrm{RHS}&=\prod_{\gb\in\sr(\mu)}\(\res(\gb)-\res(\ga)-z\)
     \prod_{\gma\in\sa(\mu)}\(\res(\gma)-\res(\ga)-z\)^{-1}\\
    &=\frac{1}{\mu_{1}+\imath-\jmath-z}\prod_{k=1}^p\frac{\mu_{i_k}-i_k+\imath-\jmath-z}
     {\mu_{i_{k+1}}-i_k+\imath-\jmath-z}.
   \end{align*}

 Suppose that $i_q<\imath\leq i_{q+1}$ for some $0\leq q\leq p+1$ where $i_{p+1}=+\infty$.
  Since $\ga=(\imath, \jmath)$ is a removable node of $\gl$ and $\nu=\gl-\{\ga\}$, we yield that  \begin{align*}
   &\gl_i=\nu_i  \text{ for } i\neq \imath;&& \hat{\gl}_j=\hat{\nu}_j \text{ for }j\neq \jmath;
    &&\gl_{\imath}=\nu_{\imath}+1=\jmath;&& \hat{\gl}_{\jmath}=\hat{\nu}_{\jmath}+1=\imath.
  \end{align*}Therefore
    \begin{align*}
   \frac{\Theta_{\gl,\mu}(x,y)}{\Theta_{\nu,\mu}(x,y)}
   &=\frac{1}{h^{\mu,\gl}_{\imath,\jmath}-z}
      \prod_{(i,j)\in\mu}\frac{h^{\nu,\mu}_{i,j}+z}{h^{\gl,\mu}_{i,j}+z}\prod_{(i,j)\in\nu}
      \frac{h^{\mu,\nu}_{i,j}-z}{h^{\mu,\gl}_{i,j}-z}\\
 &=\frac{1}{\mu_{\imath}-\jmath-z+1}\prod_{(i,j)\in\mu}\frac{\nu_i+\hat{\mu}_j-i-j+1+z}
 {\gl_i+\hat{\mu}_j-i-j+1+z}\prod_{(i,j)\in\nu}\frac{\mu_i+\hat{\nu}_j-i-j+1-z}{\mu_i+\hat{\gl}_j-i-j+1-z}\\
 &=\frac{1}{\mu_{\imath}-\jmath-z+1}
    \prod_{(\imath,j)\in\mu}\frac{\hat{\mu}_j-j+\jmath-\imath+z}{\hat{\mu}_j-j+\jmath-\imath+1+z}
    \prod_{(i,\jmath)\in\nu}\frac{\mu_i-i+\imath-\jmath-z}{\mu_i-i+\imath-\jmath+1-z}\\
     &=\frac{1}{\mu_{\imath}-\jmath-z+1}
    \prod_{j=1}^{\mu_{i_{q+1}}}\frac{\hat{\mu}_j-j+\jmath-\imath+z}{\hat{\mu}_j-j+\jmath-\imath+1+z}
    \prod_{i=1}^{\imath-1}\frac{\mu_i-i+\imath-\jmath-z}{\mu_i-i+\imath-\jmath+1-z}\\
     &=\frac{1}{\mu_{\imath}-\jmath-z+1}
    \prod_{k=p, \cdots, q+1}\,\prod_{j=\mu_{i_{k+1}+1}}^{\mu_{i_{k}}}\frac{\hat{\mu}_j-j+\jmath-\imath+z}
    {\hat{\mu}_j-j+\jmath-\imath+1+z}\\
    &\qquad\qquad\times\(\prod_{k=1}^q\prod_{i=i_{k-1}+1}^{i_k}\frac{\mu_i-i+\imath-\jmath-z}
    {\mu_i-i+\imath-\jmath+1-z}\)
    \prod_{i=i_{q}+1}^{\imath-1}\frac{\mu_i-i+\imath-\jmath-z}{\mu_i-i+\imath-\jmath+1-z}\\
    &=\frac{1}{\mu_{i_{q+1}}-i_q+\imath-\jmath-z}
    \prod_{k=q+1}^{p}\frac{i_k-\mu_{i_{k}}+\jmath-\imath+z}{i_k-\mu_{i_{k+1}}+\jmath-\imath+z}
        \prod_{k=1}^q\frac{\mu_{i_k}-i_k+\imath-\jmath-z}{\mu_{i_k}-i_{k-1}+\imath-\jmath-z}\\
    &=\prod_{k=q+1}^{p}\frac{\mu_{i_{k}}-i_k+\imath-\jmath-z}{\mu_{i_{k}}-i_{k-1}+\imath-\jmath-z}
        \prod_{k=1}^q\frac{\mu_{i_k}-i_k+\imath-\jmath-z}{\mu_{i_k}-i_{k-1}+\imath-\jmath-z}\\
        &=\frac{1}{\mu_{1}+\imath-\jmath+z}
    \prod_{k=1}^{p}\frac{\mu_{i_{k}}-i_k+\imath-\jmath-z}{\mu_{i_{k+1}}-i_{k}+\imath-\jmath-z}.
         \end{align*}
         As a consequence, we have completed the proof.\end{proof}

The following  fact will be used in the sequence.
\begin{lemma}\label{Lemm:theta-lambda-mu}Assume that $\gl$ is a partition and that $\ga$ is a removable node of $\gl$. Let $\mu$ be the subpartition of $\gl$ by removing the node $\ga$.  Then
  \begin{align*}\frac{\prod_{(i,j)\in\mu}h^{\mu}_{i,j}}{\prod_{(i,j)\in\gl} h^{\gl}_{i,j}}&=\prod_{\gb\in\sr(\mu)}\(\res(\gb)-\res(\ga)\)
  \prod_{\ga\neq\gma\in\sa(\mu)}\(\res(\gma)-\res(\ga)\)^{-1}\\
  &=\prod_{\gb\in\sr(\mu)}\(\res(\ga)-\res(\gb)\)
  \prod_{\ga\neq\gma\in\sa(\mu)}\(\res(\ga)-\res(\gma)\)^{-1}.\end{align*}
\end{lemma}

\begin{proof} Suppose that $\ga=(\imath,\jmath)$ and that \begin{align*}
    \mu=\(\mu_1=\cdots=\mu_{i_1}>\mu_{i_1+1}=\cdots=\mu_{i_2}>\cdots\cdots>\mu_{i_{p-1}+1}=\cdots=\mu_{i_p}>0\).
  \end{align*}
 We set $i_0=0$, $\mu_{i_{p+1}}=0$. Therefore we have 
  \begin{align*}
   \sr(\mu)&=\left\{(i_k,\mu_{i_k})\mid k=1,\cdots, p \right\}; \\ \sa(\mu)&=\left\{(i_k+1, \mu_{i_{k+1}}+1)\mid k=0, \cdots, p\right\}.\end{align*}
  Since $\ga$ is removable, $\imath=i_q+1$ and $\jmath=\mu_{i_q+1}+1$ for some $0\leq q\leq p+1$. Furthermore, we yield that
   \begin{align*}
   &\gl_i=\mu_i  \text{ for } i\neq \imath;&& \hat{\gl}_j=\hat{\mu}_j \text{ for }j\neq \jmath;
    &&\gl_{\imath}=\mu_{\imath}+1=\jmath;&& \hat{\gl}_{\jmath}=\hat{\mu}_{\jmath}+1=\imath.
  \end{align*}
   As a consequence, 
          \begin{align*}
   \frac{\prod_{(i,j)\in\mu}h^{\mu}_{i,j}}
   {\prod_{(i,j)\in\gl} h^{\gl}_{i,j}}
   &=\frac{\prod_{(i,j)\in\mu}\mu_i+\hat{\mu}_j-i-j+1}
             {\prod_{(i,j)\in\gl}\gl_i+\hat{\gl}_j-i-j+1}\\
  &=\prod_{(i,\jmath)\in\mu}\frac{\mu_i-i+\imath-\jmath}{\mu_i-i+\imath-\jmath+1}
      \prod_{(\imath,j)\in\mu}\frac{\hat{\mu}_j-j+\jmath-\imath}{\hat{\mu}_j-j+\jmath-\imath+1}\\
  &=\prod_{i=1}^{i_q}\frac{\mu_i-i+\imath-\jmath}{\mu_i-i+\imath-\jmath+1}
      \prod_{j=1}^{\mu_{i_{q+1}}}\frac{\hat{\mu}_j-j+\jmath-\imath}{\hat{\mu}_j-j+\jmath-\imath+1}\\
  &=\biggl(\prod_{k=1}^q\prod_{i=i_{k-1}+1}^{i_k}\frac{\mu_i-i+\imath-\jmath}
    {\mu_i\!-\!i\!+\!\imath\!-\!\jmath\!+1}\biggr)\!\biggl(\prod_{k=q+1}^p\,\prod_{j=\mu_{i_{k+1}+1}}^{\mu_{i_{k}}}\frac{\hat{\mu}_j-j+\jmath-\imath}
    {\hat{\mu}_j\!-\!j\!+\!\jmath\!-\!\imath\!+\!1}\biggr)\\
 &=\prod_{k=1}^q\frac{\mu_{i_k}-i_k+\imath-\jmath}{\mu_{i_k}-i_{k-1}+\imath-\jmath}
     \prod_{k=q+1}^{p}\frac{i_k-\mu_{i_{k}}+\jmath-\imath}{i_k-\mu_{i_{k+1}}+\jmath-\imath}\\
 &=\prod_{k=1}^q\frac{\mu_{i_k}-i_k+\imath-\jmath}{\mu_{i_k}-i_{k-1}+\imath-\jmath}
    \prod_{k=q+1}^{p}\frac{\mu_{i_{k}}-i_k+\imath-\jmath}{\mu_{i_{k+1}}-i_{k}+\imath-\jmath}\\
    &=\prod_{\gb\in\sr(\mu)}\(\res(\gb)-\res(\ga)\){
     \prod_{\ga\neq\gma\in\sa(\mu)}\(\res(\gma)-\res(\ga)\)^{-1}}\\
     &=\prod_{\gb\in\sr(\mu)}\(\res(\ga)-\res(\gb)\){
  \prod_{\ga\neq\gma\in\sa(\mu)}\(\res(\ga)-\res(\gma)\)^{-1}}.
         \end{align*}
         As a consequence, we have completed the proof.
   \end{proof}
The following result gives another combinatorial reformulation for the rational function $\Theta_{\bgl}(Q)$ defined in \S\ref{Equ:theta-bgl}.

\begin{proposition}Assume that $\bgl$ is an $m$-multipartition of $n$ and that $\ft$ is a standard $\bgl$-tableau with the node $\ga$ containing the number $n$. Let $\bmu$ be the shape of the subtableau of $\ft$ by removing the node $\ga$. Then  
  \begin{equation*}
    \Theta_{\bgl}(Q)\Theta_{\bmu}(Q)^{-1}=\prod_{\gb\in\sr(\bmu)}\(\res(\gb)-\res_{\ft}(n)\)\prod_{\ga\neq \gma\in \sa(\bmu)}\(\res(\gma)-\res_{\ft}(n)\)^{-1}.  \end{equation*}
    \label{Prop:Theta-bgl-bmu} \end{proposition}
\begin{proof}Suppose that $\bgl=(\gl^1; \dots;\gl^m )$ and $\ga=(\imath, \jmath,c)$. Then $\ga$ is removable and \begin{align*}
 \bmu=(\mu^1;\dots; \mu^m)=(\gl^1;\dots;\gl^{c-1}; \mu^c; \gl^{c+1};\dots; \gl^m),
\end{align*}
where $\mu^c=\(\gl^c_{1}, \dots, \gl^c_{\imath-1}, \gl^c_{\imath}-1,\gl^c_{\imath+1},\dots,\gl^c_{\ell(\gl^c)}\)$. Observe that $\sa(\mu^t)=\sa(\gl^t)$ and $\sr(\mu^t)=\sr(\gl^t)$ for $1\leq t\neq c\leq m$.

Applying the equality (\ref{Equ:theta-bgl}), we obtain that
  \begin{align*}
   \frac{\Theta_{\bgl}(Q)}{\Theta_{\bmu}(Q)}&=\frac{\displaystyle\prod_{1\leq s\leq m}\prod_{(i,j)\in\mu^s}\prod_{1\leq t\leq m}\(h^{\mu^s,\mu^t}_{i,j}+q_s-q_t\)}{\displaystyle\prod_{1\leq s\leq m}\prod_{(i,j)\in\gl^s}\prod_{1\leq t\leq m}\(h^{\gl^s,\gl^t}_{i,j}+q_s-q_t\)}\\
     &=\frac{\prod_{(i,j)\in\mu^c}h^{\mu^c}_{i,j}}{\prod_{(i,j)\in\gl^c}h^{\gl^c}_{i,j}}\,\prod_{\substack{1\leq t\leq m\,\&\,t\neq c}}\frac{\Theta_{\gl^c,\mu^t}(q_c, q_t)}{\Theta_{\mu^c,\mu^t}(q_c, q_t)}\\
   &=\frac{\prod_{\gb\in\sr(\mu^c)}\(\res(\gb)-\res(\ga)\)}
   {\prod_{\ga\neq\gma\in\sa(\mu^c)}\(\res(\gb)-\res(\ga)\)}\prod_{\substack{1\leq t\leq m\,\&\,t\neq c}}\frac{\prod_{\gb\in\sr(\mu^t)}\(\res(\gb)-\res(\ga)\)}
   {\prod_{\gma\in\sa(\mu^t)}\(\res(\gb)-\res(\ga)\)}\\
   & =\prod_{\gb\in\sr(\bmu)}\(\res(\gb)-\res(\ga)\)\prod_{\ga\neq \gma\in \sa(\bmu)}\(\res(\gma)-\res(\ga)\)^{-1},
  \end{align*}
  where the third equality follows by applying Lemmas~\ref{Lemm:Theta-gl-mu=nu} and \ref{Lemm:theta-lambda-mu}. We have completed the proof.
\end{proof}

\section{Fusion formulae for primitive idempotents}\label{Sec:Fusion-formulae}
From now on, the following notations will be used throughout unless otherwise stated.
\begin{notations}\label{Subsec:notations} $\bgl=(\gl^1;\cdots;\gl^m)$ is an $m$-multipartition of $n$ and $\ft$ is a standard $\bgl$-tableau with the number $n$ appearing in the node $\ga=(\imath,\jmath,c)$. We let $\fu$ be the subtableau of $\ft$ by removing the node $\ga$ and let $\bmu=\shape(\fu)$. Denote by $\fv$ the subtableau of $\fu$ which contains the numbers $1, \cdots, n-1$ and denote by $\bnu=\shape(\fv)$.
 \end{notations}

Let $\bgl$ be an $m$-multipartition of $n$ and let $S^{\bgl}$ be the Specht module corresponding to $\bgl$.  Then $S^{\bgl}$ admits the following decomposition, as a vector space,
 \begin{align*}
   S^{\bgl}=\bigoplus_{\ft\in\std(\bgl)}v_{\ft},
 \end{align*}which is equipped with a Young seminorm form and the Jucys-Murphy elements act
 diagonally in this basis.

If $\h$ is generic then $\{S^{\bgl}\mid\bgl \in \mpn\}$ is a complete set of pairwise
non-isomorphic irreducible $\h$-modules. For a standard $\bgl$-tableau $\ft$,  we denote by $E_{\ft}$ the corresponding primitive idempotent of $\h$. For all $i=1, \cdots, n$, we have
\begin{equation}\label{Equ:JiE-ft}
  J_iE_{\ft}=E_{\ft}J_i=\res_{\ft}(i)E_{\ft}.
\end{equation}
Note that the idempotent $E_{\ft}$ can be expressed in terms of the Jucys-Murphy elements.  Indeed, the inductive formula for $E_{\ft}$ in terms of Jucys-Murphy elements can be formulated as following:
\begin{equation*}
  E_{\ft}=E_{\fu}\prod_{\substack{\ga\neq \gb\in\sa\left(\bmu\right)}}\displaystyle\frac{J_n-\res(\gb)}{\res_{\ft}(n)-\res(\gb)},
\end{equation*}
with the initial condition $E_{\emptyset}=1$, which is well-defined thanks to Lemma~\ref{Lemm:Separation}.

 On the other hand, let $\ft_1, \cdots, \ft_a$ be the set of pairwise different standard $\bgl$-tableaux obtained for $\fu$ by adding an node with number $n$. Then the branching properties of the Young basis imply that
\begin{equation*}E_{\fu}=\sum_{i=1}^aE_{\ft_i}\end{equation*} and the rational function $E_{\fu}\frac{z-\res_{\ft}(n)}{z-J_n}$ in $z$ is well-defined, which is non-singular at $z=\res_{\ft}(n)$ according to (\ref{Equ:JiE-ft}). Furthermore, we have
\begin{equation}\label{Equ:E_ft=E_fu}
  E_{\ft}=\biggl.E_{\fu}\frac{z-\res_{\ft}(n)}{z-J_n}\biggr|_{z=\res_{\ft}(n)}.
\end{equation}

 We first define the following rational function in variable $z$:

  \begin{equation}\label{Equ:theta-def}
    \Theta_{\ft}(z):=\frac{z-\res_{\ft}(n)}{f(z)}\prod_{i=1}^{n-1}\frac{{\(}z-\res_{\ft}(i){\)}^2}
    {\(z-\res_{\ft}(i)+1\)\(z-\res_{\ft}(i)-1\)}.
  \end{equation}

Clearly if $n=1$ then $\Theta_{\ft}(z)=\frac{z-\res_{\ft}(1)}{f(z)}$.

\begin{lemma}\label{Lemm:theta-z}
  Keep notations as in \S\ref{Subsec:notations}. Then
  \begin{equation*}\label{Equ:theta-z}
    \Theta_{\ft}(z)=\(z-\res_{\ft}(n)\)\prod_{\gb\in\sr(\bmu)}\(z-\res(\gb)\)\prod_{\gma\in \sa(\bmu)}\(z-\res(\gma)\)^{-1}.  \end{equation*}
\end{lemma}
\begin{proof}
  We prove  the lemma by induction on $n$. If $n=1$ then $\bmu=\emptyset$, $\sr(\bmu)=\emptyset$, and $\sa(\bmu)=\{(1,1,i)|1\leq i\leq m\}$. Thus the lemma follows directly by using the equality (\ref{Equ:theta-def}) for $n=1$.

  Now assume that the lemma holds for all standard tableaux $\ft$ with $n-1\geq 1$ nodes. We show that it also holds for the standard tableaux with $n$ nodes. Suppose that the node $\ga=(a,b,c)$ of $\ft$ contains the number $n-1$. Then we have the following cases:
\noindent  \begin{enumerate}
    \item If $(a-1,b,c), (a,b-1,c)\notin \sr(\bnu)$ then $\sr(\bmu)=\sr(\bnu)\cup\{\ga\}$ and
       \begin{align*}\sa(\bmu)=\(\sa(\bnu)\cup\{(a+1,b,c),(a,b+1,c)\}\)\setminus \{\ga\}.\end{align*}

   \item If $(a-1,b,c)\in \sr(\bnu)$ and $(a,b-1,c)\notin\sr(\bnu)$, then \begin{align*}&\sr(\bmu)=\(\sr(\bnu)\cup\{(a,b,c)\}\)\setminus \{(a,b,c)\};\\ &\sa(\bmu)=\(\sa(\bnu)\cup\{(a+1,b,c)\}\)\setminus \{\ga\}.\end{align*}
  \item If $(a-1,b,c)\notin \sr(\bnu)$ and $(a,b-1,c)\in\sr(\bnu)$, then \begin{align*}&\sr(\bmu)=\(\sr(\bnu)\cup\{\ga\}\)\setminus \{(a,b-1,c)\};\\ &\sa(\bmu)=\(\sa(\bnu)\cup\{(a+1,b,c)\}\)\setminus \{\ga\}.\end{align*}
   \item If $(a-1,b,c), (a,b-1,c)\in\sr(\bnu)$, then $\sa(\bmu)=\sa(\bnu)\setminus \{\ga\}$ and \begin{align*}\sr(\bmu)=\(\sr(\bnu)\cup\{\ga\}\)\setminus \{(a-1,b,c), (a,b-1,c)\}.\end{align*}
  \end{enumerate}
 Now applying the induction argument, we obtain
  \begin{align*}
    \Theta_{\ft}(z)&=\frac{(z-\res_{\ft}(n))(z-\res_{\ft}(n-1))^2}{(z-\res_{\ft}(n)+1)(z-\res_{\ft}(n)-1)}
    \prod_{\gb\in\sr(\bnu)}\(z-\res(\gb)\)\prod_{\gma\in\sa(\bnu)}\(z-\res(\gma)\)^{-1}\\
    &=\(z-\res_{\ft}(n)\)\prod_{\gb\in\sr(\bmu)}\(z-\res(\gb)\)\prod_{\gma\in \sa(\bmu)}\(z-\res(\gma)\)^{-1}.
  \end{align*}
  We complete the proof.
\end{proof}

The following lemma establishes the relationship between the rational function $\Theta_{\bgl}(Q)$ and the rational function $\Theta_{\ft}(z)$, which is crucial to the fusion formula.
\begin{lemma}\label{Lemm:Theta_ft-bgl-bmu}
  The rational function $\Theta_{\ft}(z)$ is non-singular at $z=\res_{\ft(n)}$ and $$\Theta_{\ft}\(\res_{\ft}(n)\)=\Theta_{\bgl}(Q)\Theta_{\bmu}(Q)^{-1}.$$
\end{lemma}

\begin{proof}Lemma~\ref{Lemm:theta-z} shows that the rational function $\Theta_{\ft}(z)$ is non-singular at $z=\res_{\ft}(n)$. Noticing that the node $\ga$ is removable and applying Proposition~\ref{Prop:Theta-bgl-bmu}, we obtain that
\begin{align*}
  \Theta_{\ft}\(\res_{\ft}(n)\)&=\prod_{\gb\in\sr(\bmu)}\(\res_{\ft}(n)-\res(\gb)\)\prod_{\ga\neq\gma\in\sa(\bmu)
 }\(\res_{\ft}(n)-\res(\gma)\)^{-1}\\
  &=\Theta_{\bgl}(Q)\Theta_{\bmu}(Q)^{-1}.
\end{align*}
It completes the proof.\end{proof}

Let $\phi_1(z)=t(z)$ and define \begin{equation*}\begin{aligned}
  \phi_{k+1}(z_1, \cdots, z_k;z)&:=t_k(z,z_k)\phi_k(z_1, \cdots, z_{k-1};z)t_k\\
  &=t_k(z,z_k)t_{k-1}(z,z_{k-1})\cdots t_1(z,z_1)t_1\cdots t_k.\end{aligned}
\end{equation*}

\begin{lemma}\label{Lemm:Theta-phi-E}
  Keep notations as in \S\ref{Subsec:notations}. Then
  \begin{equation}\label{Equ:Theta-phi-E}
    \Theta_{\ft}(z)\phi_n\(\res_{\ft}(1), \cdots,\res_{\ft}(n-1),z\)E_{\fu}=\frac{z-\res_{\ft}(n)}{z-J_n}E_{\fu}.
  \end{equation}
\end{lemma}

\begin{proof}
  We prove the equality (\ref{Equ:Theta-phi-E}) by induction on $n$. Using (\ref{Equ:def-t(z)}), (\ref{Equ:JM-def}) and (\ref{Equ:theta-def}), we obtain
  \begin{align*}&
    \frac{z-\res_{\ft}(1)}{z-J_1}E_{\emptyset}=\frac{z-\res_{\ft}(1)}{f(z)}\cdot \frac{f(z)}{z-t_0}E_{\emptyset}=\Theta_{\ft}(z)\phi_1(z)E_{\emptyset}.
  \end{align*}
That is the equality holds for $n=1$.

Note that $E_{\fv}E_{\fu}=E_{\fu}$ and $E_{\fu}t_{n-1}=t_{n-1}E_{\fu}$.  Then, by induction hypothesis, the left hand side of the equality (\ref{Equ:Theta-phi-E}) can be rewritten as
 \begin{align*}\mathrm{LHS}&=\negmedspace\Theta_{\ft}(z)t_{n\!-\!1}\(z,\res_{\ft}( n\negmedspace-\negmedspace1)\)
 \phi_{n\!-\!1}\(\res_{\ft}(1), \cdots,\res_{\ft}( n\negmedspace-\negmedspace2),z\)t_{n-1}E_{\fv}E_{\fu}\\
 &=\negmedspace\Theta_{\ft}(z)t_{n-1}(z,\res_{\ft}(n-1))\phi_{n-1}\(\res_{\ft}(1), \cdots,\res_{\ft}(n\negmedspace-\negmedspace2),z\)
 E_{\fv}t_{n-1}E_{\fu}\\
 &=\negmedspace\frac{\Theta_{\ft}(z)}{\Theta_{\fu}(z)}t_{n-1}\(z,\res_{\ft}(n\negmedspace-\negmedspace1)\)
 \(\Theta_{\fu}(z)\phi_{n-1}(\res_{\ft}(1), \cdots,\res_{\ft}(n\negmedspace-\negmedspace2),z)E_{\fv}\)t_{n-1}E_{\fu}\\
 &=\negmedspace\Theta_{\ft}(z)\Theta_{\fu}(z)^{-1}t_{n-1}(z,\res_{\ft}(n-1))
 \frac{z-\res_{\ft}(n-1)}{z-J_{n-1}}E_{\fv}t_{n-1}E_{\fu}\\
 &=\negmedspace\Theta_{\ft}(z)\Theta_{\fu}(z)^{-1}t_{n-1}(z,\res_{\ft}(n-1))
 \frac{z-\res_{\ft}(n-1)}{z-J_{n-1}}t_{n-1}E_{\fu}\\
 &=\negmedspace\frac{(z-\res_{\ft}(n-1))^2t_{n-1}(z,\res_{\ft}(n-1))}{(z-\res_{\ft}(n-1)+1)(z-\res_{\ft}(n-1)+1)}
 \frac{z-\res_{\ft}(n)}{z-J_{n-1}}t_{n-1}E_{\fu}\\
 &=t_{n-1}(\res_{\ft}(n-1),z)^{-1}\frac{z-\res_{\ft}(n)}{z-J_{n-1}}t_{n-1}E_{\fu}.
   \end{align*}
Note that $J_n$ commutes with $E_{\fu}$. Therefore, to prove the equality (\ref{Equ:Theta-phi-E}), it suffices to show that
 \begin{align*}
   &t_{n-1}(z-J_n)E_{\fu}=(z-J_{n-1})t_{n-1}(\res_{\ft}(n-1),z)E_{\fu},
 \end{align*}
 which follows directly by using the equalities (\ref{Equ:t_{ij}(x,y)}) and (\ref{Equ:JM-def}). It completes the proof.
\end{proof}

Define the following rational function with values in the algebra $\h$:
\begin{equation}
  \Phi(z_1, \cdots, z_n):=\phi_n(z_1, \cdots,z_{n-1}, z_n)\phi_{n-1}(z_1, \cdots, z_{n-1})\cdots\phi_1(z_1).
\end{equation}

Now we can prove the main result of this paper.

\begin{theorem}\label{Them:Main}
  Let $\bgl$ be an $m$-multipartition of $n$ and $\ft$ a standard $\bgl$-tableau. Then the primitive idempotent $E_{\ft}$ of $\h_{m,n}(Q)$ corresponding to $\ft$ can be obtained by the following consecutive evaluations
  \begin{equation*}
    E_{\ft}= \Theta_{\bgl}(Q)\Phi(z_1, \cdots, z_n)\biggl.\biggr|_{z=\res_{\ft}(1)}\cdots\, \biggl. \biggr|_{z=\res_{\ft}(n\!-\!1)}\biggl.\biggr|_{z=\res_{\ft}(n)}.
  \end{equation*}
   \end{theorem}
   \begin{proof}
     The theorem follows, by induction on $n$, from (\ref{Equ:E_ft=E_fu}) and Lemmas~\ref{Lemm:Theta-phi-E} and \ref{Lemm:Theta_ft-bgl-bmu}.
   \end{proof}

We close this paper with some remarks on the study of the fusion procedure for the degenerate cyclotomic Hecke algebras and related topics.

\begin{point}{}* The fusion procedure for the Yang-Baxter equation was first introduced by Kulish et.al. in \cite{Kulish-R-S}, which allows the construction of new solutions of the Yang-Baxter equation starting from a given fundamental solution. As it is suggested in \cite{Molev}, a ``fused solution" of the Yang-Baxter equation can be investigated using the certain version of the Schur-Weyl duality (see \cite{dAndecy}). Note that Brundan and Kleshchev have established  the Schur-Weyl duality for higher levels in \cite{BK-Schur-Weyl}. It may be interesting to use the fusion formula for the degenerate cyclotomic Hecke algebras to obtain a family of fused solutions of the Yang-Baxter equation acting on finite-dimensional irreducible representations of the finite $W$-algebras.

 It is well-known that the (degenerate) cyclotomic Hecke algebras, the (degenerate) affine Hecke algebras are closely related (ref.\cite{Kbook,BK-Block}). We would like to know whether the fusion formulae in this paper are certain specializations of those ones in \cite{Nazarov-2003} and \cite{Ogievetsky-Andecy2011,Ogievetsky-Andecy2013}. On the other hand, it seems likely that we may use the reflection equation for the degenerate cyclotomic Hecke algebras to define and study the Bethe subalgebras of the degenerate cyclotomic Hecke algebras (cf. \cite{Isaev-Kirillov}).  We hope that these Bethe subalgebras can be used to investigae the center of the (degenerate) cyclotomic Hecke algebras.
\end{point}


\end{CJK*}
\end{document}